\documentclass[12pt]{article}
\usepackage[colorlinks,urlcolor=blue]{hyperref}
\newtheorem{theorem}{Theorem}
\newtheorem{lemma}{Lemma}
\newenvironment{proof}{
 \bgroup\noindent\small{\bf Proof\ }}{
 \nolinebreak\hbox{\ $\Box$}
 \egroup}

\usepackage{amsfonts}
\newcommand{\EE}{\mathbb{E}}

\usepackage{bbm}

\newcommand{\D}{{\mathrm{d}}}

\newcommand{\da}{\dot{a}}
\newcommand{\dda}{\ddot{a}}
\newcommand{\db}{\dot{b}}
\newcommand{\ddb}{\ddot{b}}
\newcommand{\dS}{\dot{S}}
\newcommand{\ddS}{\ddot{S}}

\newcommand{\dM}{\dot{M}}

\newcommand{\ha}{\widehat{a}}
\newcommand{\hda}{\widehat{\dot{a}}}
\newcommand{\hb}{\widehat{b}}
\newcommand{\hdb}{\widehat{\dot{b}}}

\newcommand{\ba}{{\bf a}}
\newcommand{\bb}{{\bf b}}
\newcommand{\bS}{{\bf S}}

\newcommand{\hS}{\widehat{S}}
\newcommand{\hdS}{\widehat{\dS}}

\newcommand{\us}{{\underline{s}}}
\newcommand{\ut}{{\underline{t}}}
\newcommand{\uu}{{\underline{u}}}

\newcommand{\fracs}[2]{{\textstyle \frac{#1}{#2}}}

\begin{document}

\title{Strong convergence of path sensitivities
  \thanks{This research was funded by the EPSRC ICONIC programme grant
          EP/P020720/1 and by the Hong Kong Innovation and Technology
          Commission (InnoHK Project CIMDA) and their support is gratefully
          acknowledged.}}

\author{M.B.~Giles}
         
\maketitle

\begin{abstract}
  It is well known that the Euler-Maruyama discretisation of
  an autonomous SDE using a uniform timestep $h$ has a strong
  convergence error which is $O(h^{1/2})$ when the drift and
  diffusion are both globally Lipschitz. This note proves that
  the same is true for the approximation of the path sensitivity
  to changes in a parameter affecting the drift and diffusion,
  assuming the appropriate number of derivatives exist and are
  bounded. This seems to fill a gap in the existing stochastic
  numerical analysis literature.
\end{abstract}

\section{Introduction}

Suppose we have an autonomous scalar SDE
\[
\D S_t = a(\theta,S_t)\, \D t + b(\theta,S_t)\, \D W_t,
\]
with given initial data $S_0$, in which the drift and diffusion
coefficients depend on a scalar parameter $\theta$ as well as
the path $S_t$.  The corresponding Euler-Maruyama discretisation,
using a fixed timestep $h$ is given by
\[
  \hS_{(n+1)h} = \hS_{nh} + a(\theta,\hS_{nh})\, h + b(\theta,\hS_{nh})\, \Delta W_n,
\]
where $\Delta W_t$ is a $N(0,h)$ random variable, and $\hS_0 = S_0$.

For a given $\theta$, if $a$ and $b$ are both globally
Lipschitz it is well known (see Theorem 10.6.3 in \cite{kp92},
and the subsequent discussion) that over a finite time interval
$[0,T]$, for any $p\geq 2$ there is a constant $c_p$ such that
\[
\EE\left[ \sup_{0<t<T} | \hS_t{-}S_t |^p \right] \leq c_p \, h^{p/2},
\]
where $\hS_t$ is the interpolation of the Euler-Maruyama
approximation defined by
\[
\D \hS_t = a(\theta,\hS_\ut)\, \D t + b(\theta,\hS_\ut)\, \D W_t,
\]
where $\ut$ represents $t$ rounded down to the nearest timestep.

If $a(\theta,S)$ is differentiable with respect to both arguments,
and we use the notation $a'\equiv \partial a/\partial S$
and $\da\equiv \partial a/\partial \theta$ then differentiating
the original SDE once w.r.t.~$\theta$ gives the linear pathwise
sensitivity SDE for $\dS_t \equiv \partial S_t/\partial \theta$,
\[
  \D \dS_t = (\da(\theta,S_t) + a'(\theta,S_t)\, \dS_t)\, \D t
  + (\db(\theta,S_t) + b'(\theta,S_t)\, \dS_t)\, \D W_t.
\]
It is easily seen that the Euler-Maruyama discretisation of this
SDE
\[
  \hdS_{(n+1)h} = \hdS_{nh} +
  \left(\da(\theta,\hS_{nh}) + a'(\theta,\hS_{nh})\, \hdS_{nh} \right) h
+ \left(\db(\theta,\hS_{nh}) + b'(\theta,\hS_{nh})\, \hdS_{nh} \right) \Delta W_n
\]
corresponds to the differentiation of the Euler-Maruyama
discretisation of the original SDE. This is used extensively in the
computational finance community as part of the pathwise sensitivity
approach (also known as IPA, Infinitesimal Perturbation Analysis)
to computing payoff sensitivities known collectively as ``the Greeks''
\cite{bg96,cg24,gg06,glasserman04,ecuyer90}.

The two SDEs can be combined to form a single vector SDE
\[
\D \bS_t = \ba(\theta, \bS_t)\, \D t + \bb(\theta, \bS_t)\, \D W_t.
\]
From this it seems natural that the path sensitivity approximation
should have the usual half order strong convergence, which is very
important for its use and analysis in the context of multilevel
Monte Carlo methods \cite{burgos14,bg12,giles15}.
However, there is a problem; except in very simple cases, $\ba$ and
$\bb$ do not satisfy the usual global Lipschitz condition since
\[
  b'(\theta,v_1)\, v_2 - b'(\theta,u_1)\, u_2
   = (b'(\theta,v_1) - b'(\theta,u_1))\, v_2 + b'(\theta,u_1)\, (v_2-u_2),
\]
so $b'(\theta,v_1)\, v_2$ is not uniformly Lipschitz when $u_2{=}v_2$
unless $b'(\theta,u_1) {=} b'(\theta,v_1)$ for all $\theta$, $u_1$, $v_1$.

In this note we prove that, despite this,  $O(h^{1/2})$ strong
convergence is achieved by $\hdS_t$, the Euler-Maruyama approximation
to $\dS_t$, and the same holds for higher derivatives, and for cases
in which $S_t$ and $\theta$ are vector quantities.

The proof comes from re-tracing the steps of the analysis in \cite{kp92}
which prove that for a finite time interval $[0,T]$ and any $p\geq 2$
there exist constants $c_p^{(1)}, c_p^{(2)}, c_p^{(3)}$, such that
\begin{eqnarray*}
  \EE\left[\sup_{0<t<T} |S_t|^p \right] &\leq & c_p^{(1)},   \\[0.1in]
  \EE\left[\, |S_t {-} S_{t_0}|^p \right] &\leq & c_p^{(2)}\, (t{-}t_0)^{p/2},
  \mbox{~~ for any } 0<t_0<t<T,  \\[0.1in]
  \EE\left[ \sup_{0<t<T} |\hS_t {-} S_t|^p \right] &\leq & c_p^{(3)}\, h^{p/2},
\end{eqnarray*}
proving that corresponding results hold for $\dS_t$ and $\hdS_t$,
primarily because of the boundedness of $a'$ and $b'$ which
multiply $\dS_t$ in the drift and diffusion coefficients.

\section{SDE sensitivity analysis}

For the first order sensitivity analysis we assume that the first
derivatives $a'$, $\da$, $b'$, $\db$ and the second derivatives
$a''$, $\da'$, $\dda$, $b''$, $\db'$, $\ddb$ all exist and are
uniformly bounded so that there exist constants
$L_a, L_b$ such that
\begin{eqnarray*}
\sup_{\theta, S} \max\left\{ |a'(\theta,S)|,\ |\da(\theta,S)|,\ 
|a''(\theta,S)|,\ |\da'(\theta,S)|,\ |\dda(\theta,S)| \right\} &\leq& L_a,\\
\sup_{\theta, S} \max\left\{ |b'(\theta,S)|,\ |\db(\theta,S)|,\ 
|b''(\theta,S)|,\ |\db'(\theta,S)|,\ |\ddb(\theta,S)| \right\} &\leq& L_b.
\end{eqnarray*}

The pathwise sensitivity SDE is
\[
\D \dS_t = (\da_t + a_t' \dS_t)\, \D t + (\db_t + b_t' \dS_t)\, \D W_t,
\]
subject to initial data $\dS_0$ which may be non-zero if $S_0$ also
depends on $\theta$.  Here we use the notation $\da_t$ to represent
$\da(\theta,S_t)$, with a similar interpretation for $a_t'$, $\db_t$
and $b_t'$.
\begin{lemma}
For a given time interval $[0,T]$, and any $p\geq 2$,
there exists a constant $c_p^{(1)}$ such that
\[
\sup_{0<t<T} \EE\left[\, |\dS_t|^p \right] \leq  c_p^{(1)}.
\]
\end{lemma}

\begin{proof}
For even integer $p\geq 2$, if we define $P_t = \dS_t^p$ then Ito's lemma
gives us
\[
  \D P_t = \left( p\, \dS_t^{p-1} (\da_t + a_t' \dS_t)
    + \fracs{1}{2} \, p\,(p{-}1)\, \dS_t^{p-2}(\db_t + b_t' \dS_t)^2 \right) \D t
  + p\, \dS_t^{p-1} (\db_t + b_t' \dS_t)\, \D W_t.
\]
Using the fact that $|\dS_t|^q < 1 + \dS_t^p$ for $0<q<p$, and
$(\db_t + b_t' \dS_t)^2 \leq 2\, \db_t^2 + 2\,(b_t' \dS_t)^2$ we obtain
\[
\D \EE[P_t] \leq (p\,L_a + p\,(p{-}1)\,L_b^2)\, (1 + 2\, \EE[P_t])\, \D t
\]
and hence $\EE[P_t]$ is finite over $[0,T]$ by Gr\"onwall's inequality.

Bounds for other values of $p$ can be obtained using H\"older's inequality.
\end{proof}

The previous result is strengthened in the following lemma.

\begin{theorem}
\label{thm:moments}
For a given time interval $[0,T]$, and any $p\geq 2$,
there exists a constant $c^{(1)}_p$ such that
\[
\EE\left[ \sup_{0<t<T} | \dS_t |^p \right] \leq c^{(1)}_p.
\]
\end{theorem}

\begin{proof}
Starting from
\[
  \dS_t = \dS_0 + \int_0^t (\da_s + a_s' \dS_s)\, \D s
                + \int_0^t (\db_s + b_s' \dS_s)\, \D W_s,
\]
and defining
\[
\dM^{(p)}_t = \EE\left[ \sup_{0<s<t} | \dS_s |^p \right],
\]
Jensen's inequality gives
\begin{eqnarray*}
\dM_t^{(p)} &\leq& 5^{p-1} \left(\ |\dS_0|^p + 
   \EE\left[ \sup_{0<s<t} \left| \int_0^s  \da_u \, \D u \right|^p  \right]
               + \EE\left[ \sup_{0<s<t}  \left| \int_0^s  a_u' \dS_u \, \D u \right|^p  \right]
               \right. \\ && \left. ~~~~~
 +\, \EE\left[ \sup_{0<s<t} \left| \int_0^s  \db_u \, \D W_u \right|^p  \right]
 + \EE\left[ \sup_{0<s<t}  \left| \int_0^s   b_u' \dS_u \, \D W_u \right|^p \right] \ \right).
\end{eqnarray*}
Jensen's inequality for integrals gives
\[
  \left| \int_0^s  \da_u \, \D u \right|^p
 \ \leq\ s^{p-1} \int_0^s  |\da_u|^p\, \D u
 \ \leq\ t^{p-1} \int_0^t  |\da_u|^p\, \D u,
\]
for $0\leq s\leq t \leq T$, and hence
\[
   \EE\left[ \sup_{0<s<t}  \left| \int_0^s  \da_u \, \D u \right|^p \right]
  \ \leq\ t^{p-1} \int_0^t  \EE\left[\,  |\da_u|^p \right] \, \D u
  \ \leq\ L_a^p\ t^p.
\]
Similarly,
\[
  \EE\left[ \sup_{0<s<t}  \left| \int_0^s  a_u' \dS_u \, \D u \right|^p \right]
  \ \leq\ t^{p-1} \int_0^t  L_a^p\ \EE\left[ |\dS_u|^p\right] \, \D u
  \ \leq\ L_a^p\ t^{p-1} \int_0^t  \dM_u^{(p)} \, \D u.
\]
The BDG (Burkholder-Davis-Gundy) inequality \cite{bdg72} gives
\[
  \EE\left[ \sup_{0<s<t} \left| \int_0^s  \db_u \, \D W_u \right|^p  \right]
 \, \leq\, C_p\  \EE\left[ \left( \int_0^t  |\db_u|^2  \, \D u \right)^{p/2}\right]
 \ \leq\ C_p\, L_b^p\ t^{p/2}
\]
where $C_p$ is a constant arising from the BDG inequality, 
and similarly
\begin{eqnarray*}
  \EE\left[ \sup_{0<s<t} \left| \int_0^s  b_u' \dS_u \, \D W_u \right|^p  \right]
   &\leq& C_p \  \EE\left[ \left( \int_0^t  L_b^2\, | \dS_u |^2 \, \D u \right)^{p/2} \right]
\\ &\leq& C_p\, L_b^p \ t^{p/2-1} \int_0^t  \EE\left[ | \dS_u |^p\right]  \, \D u
\\ &\leq& C_p\, L_b^p \ t^{p/2-1} \int_0^t  \dM_u^{(p)} \, \D u.
\end{eqnarray*}
Combining these bounds, and noting that $t\leq T$, we obtain constants $c_1, c_2$ for which
\[
   \dM_t^{(p)} \leq c_1 + c_2 \int_0^t  \dM_u^{(p)}\, \D u
\]
and the desired bound for $\dM_t^{(p)}$ follows from Gr\"onwall's inequality.
\end{proof}

\begin{lemma}
\label{lemma:short_time}
For a given time interval $[0,T]$, and any $p\geq 2$,
there exists a constant $c_p^{(2)}$ such that
\[
\EE\left[\, |\dS_t-\dS_{t_0}|^p \right] \leq  c_p^{(2)} (t{-}t_0)^{p/2}
\]
for any $0\leq t_0 \leq t\leq T$.
\end{lemma}
\begin{proof}
The proof is almost identical to the previous proof, but starting from
\[
  \dS_t - \dS_{t_0} = \int_{t_0}^t \left(\da_s + a_s' \dS_{t_0} + a_s' (\dS_s {-} \dS_{t_0})\right)\, \D s
                   + \int_{t_0}^t \left(\db_s + b_s' \dS_{t_0} + b_s' (\dS_s {-} \dS_{t_0})\right)\, \D W_s,
\]
and defining 
\[
\dM^{(p)}_t = \EE\left[ \sup_{t_0<s<t} | \dS_s - \dS_{t_0} |^p \right],
\]
leading to there being constants $c_1, c_2$ such that
\[
 \dM_t^{(p)} \leq c_1 (t{-}t_0)^{p/2} + c_2 \int_{t_0}^t  \dM_u^{(p)}\, \D u.
\]
The result then follows again from Gr\"onwall's inequality.
\end{proof}

\section{Strong convergence analysis}

The integral form of the SDE for the first order sensitivity is
\[
  \dS_t = \dS_0 + \int_0^t (\da_s + a_s' \dS_s)\, \D s
  + \int_0^t (\db_s + b_s' \dS_s)\, \D W_s,
\]
and the corresponding continuous Euler-Maruyama discretisation
can be defined as
\[
  \hdS_t = \hdS_0 + \int_0^t (\hda_\us + \ha_\us' \hdS_\us)\, \D s
  + \int_0^t (\hdb_\us + \hb_\us' \hdS_\us)\, \D W_s,
\]
where the notation $\us$ denotes $s$ rounded downwards to the
nearest timestep, and $\hda_\us$ denotes $\da(\theta, \hS_\us)$
with similar meanings for $\ha'_\us$, $\hdb_\us$ and  $\hb'_\us$.

\begin{lemma}
  \label{lemma:moments}
For a given time interval $[0,T]$, and any $p\geq 2$,
there exists a constant $c^{(1)}_p$ such that
\[
\EE\left[ \sup_{0<t<T} | \hdS_t |^p \right] \leq c^{(1)}_p.
\]
\end{lemma}
\begin{proof}
The proof follows the same approach used with Theorem \ref{thm:moments}.
\end{proof}

We now come to the strong convergence theorem.

\begin{theorem}
Given the assumption about the boundedness of all first and second
derivatives, for a given time interval $[0,T]$, and any $p\geq 2$,
there exists a constant $c^{(3)}_p$ such that
\[
\EE\left[ \sup_{0<t<T} | \hdS_t - \dS_t |^p \right] \leq c^{(3)}_p h^{p/2}.
\]
\end{theorem}

\begin{proof}
Defining $E_t = \hdS_t - \dS_t$, the difference between the two is
\begin{eqnarray*}
  E_t  &=& \int_0^t (\hda_\us{-}\da_s) + (\ha_\us' \hdS_\us {-} a_s'\dS_s)\, \D s
    + \int_0^t (\hdb_\us{-}\db_s) + (\hb_\us' \hdS_\us{-} b_s'\dS_s)\, \D W_s \\
         &=& \int_0^t (\hda_\us{-}\da_\us) + (\ha_\us' \hdS_\us {-} a_\us'\dS_\us)\,
             + (\da_\us{-}\da_s) + (a_\us' \dS_\us{-} a_s'\dS_s)\, \D s \\
         &+& \int_0^t (\hdb_\us{-}\db_\us) + (\hb_\us' \hdS_\us{-} b_\us'\dS_\us)
             + (\db_\us{-}\db_s) + (b_\us' \dS_\us{-} b_s'\dS_s)\, \D W_s \\
         &=& \int_0^t (\hda_\us{-}\da_\us) + (\ha_\us'{-} a_\us') \hdS_\us
    + (\da_\us{-}\da_s) + (a_\us' {-} a_s')\dS_\us + a_s' (\dS_\us {-} \dS_s)\, \D s \\
         &+& \int_0^t (\hdb_\us{-}\db_\us) + (\hb_\us'{-} b_\us') \hdS_\us
    + (\db_\us{-}\db_s) + (b_\us' {-} b_s')\dS_\us + b_s' (\dS_\us {-} \dS_s)\, \D W_s \\
         &+& \int_0^t a_\us' E_\us \, \D s  + \int_0^t b_\us' E_\us \, \D W_s.
\end{eqnarray*}

This gives us 12 terms to bound, 5 from the first integral, 5 from the
second integral, and 2 from the last two integrals in the above expression.

For the first pair, given that all second derivatives of $a$ are bounded by
$L_a$, we have
\[
\EE\left[\sup_{0<s<t} \left| \int_0^s (\hda_\uu{-}\da_\uu)\, \D  u \right|^p  \right]
\leq T^{p-1}  \int_0^T \EE[\, |\hda_\uu{-}\da_\uu|^p] \, \D u
\leq L_a^p\, T^{p-1}  \int_0^T \EE[\, |\hS_\uu{-}S_\uu|^p] \, \D u,
\]
and similarly, using the BDG inequality, 
\[
\EE\left[\sup_{0<s<t} \left| \int_0^s (\hdb_\uu{-}\db_\uu)\, \D  W_u \right|^p  \right]
\leq C_p\, \EE\left[ \left( \int_0^t |\hdb_\uu{-}\db_\uu|^2 \, \D u \right)^{p/2}\right]
\leq C_p\, L_b^p\, T^{p/2-1} \! \int_0^T \EE[\, |\hS_\uu{-}S_\uu|^p] \, \D u.
\]
For the second pair we need to also use H\"older's inequality to give
\[
  \EE\left[\sup_{0<s<t} \left| \int_0^s (\ha_\uu'{-} a_\uu') \hdS_\uu\, \D u \right|^p  \right]
  \leq L_a^p\, T^{p-1}  \int_0^T \EE[\, |\hS_\uu{-}S_\uu|^{2p}]^{1/2}\ \EE[\, |\hdS_\uu|^{2p}]^{1/2}  \, \D u,
\]
and
\[
  \EE\left[\sup_{0<s<t} \left| \int_0^s (\hb_\uu'{-} b_\uu') \hdS_\uu\, \D W_u \right|^p  \right]
  \leq C_p\, L_b^p\, T^{p/2-1}  \int_0^T \EE[\, |\hS_\uu{-}S_\uu|^{2p}]^{1/2}\ \EE[\, |\hdS_\uu|^{2p}]^{1/2}  \, \D u.
\]
Similarly, for the third pair we have
\[
  \EE\left[\sup_{0<s<t} \left| \int_0^s (\da_\uu{-} \da_u) \, \D u \right|^p  \right]
  \leq L_a^p\, T^{p-1}  \int_0^T \EE[\, |S_\uu{-}S_u|^p] \, \D u,
\]
and
\[
  \EE\left[\sup_{0<s<t} \left| \int_0^s (\db_\uu{-} \db_u) \, \D W_u \right|^p  \right]
  \leq C_p\, L_b^p\, T^{p/2-1}  \int_0^T \EE[\, |S_\uu{-}S_u|^p] \, \D u,
\]
for the fourth pair we have
\[
  \EE\left[\sup_{0<s<t} \left| \int_0^s (a'_\uu{-} a'_u)\,\dS_\uu \, \D u \right|^p  \right]
  \leq L_a^p\, T^{p-1}  \int_0^T \EE[\, |S_\uu{-}S_s|^{2p}]^{1/2}
                            \, \EE[\, |\dS_\uu|^{2p}]^{1/2} \, \D u,
\]
and
\[
  \EE\left[\sup_{0<s<t} \left| \int_0^s (b'_\uu{-} b'_u)\,\dS_\uu \, \D W_u \right|^p  \right]
  \leq C_p\, L_b^p\, T^{p/2-1}  \int_0^T  \EE[\, |S_\uu{-}S_u|^{2p}]^{1/2}
                            \, \EE[\, |\dS_\uu|^{2p}]^{1/2} \, \D u,
\]
and for the fifth pair we have
\[
  \EE\left[\sup_{0<s<t} \left| \int_0^s a'_u (\dS_\uu{-}\dS_u)  \, \D u \right|^p  \right]
  \leq L_a^p\, T^{p-1}  \int_0^T \EE[\, |\dS_\uu{-}\dS_u|^p] \, \D u,
\]
and
\[
  \EE\left[\sup_{0<s<t} \left| \int_0^s  b'_u (\dS_\uu{-}\dS_u) \, \D W_u \right|^p  \right]
  \leq C_p\, L_b^p\, T^{p/2-1}  \int_0^T  \EE[\, |\dS_\uu{-}\dS_u|^p] \, \D u.
\]

For the final pair we have
\[
  \EE\left[\sup_{0<s<t} \left| \int_0^s a'_\uu E_\uu \, \D u \right|^p  \right]
  \leq L_a^p\, T^{p-1}  \int_0^t \EE\left[ \sup_{0<u<s} |E_u|^p \right] \, \D s,
\]
and
\[
  \EE\left[\sup_{0<s<t} \left| \int_0^s  b'_\uu E_\uu \, \D W_u \right|^p  \right]
  \leq C_p\, L_b^p\, T^{p/2-1}  \int_0^t \EE\left[ \sup_{0<u<s} |E_u|^p \right] \, \D s.
\]

Since $\EE[\, |\hS_\us{-}S_\us|^p]$ and $\EE[\, |S_\us{-}S_s|^p]$
are both $O(h^{p/2})$ due to standard results, and
$\EE[\, |\dS_\us{-}\dS_s|^p]$ is $O(h^{p/2})$ due to
Lemma \ref{lemma:short_time}, and
$\EE[\, |\dS_\us|^p]$ and $\EE[\, |\hdS_\us|^p]$ are both finite
due to Theorem \ref{thm:moments} and Lemma \ref{lemma:moments},
it follows that there are
constants $c_1, c_2$ such that for $0\leq t\leq T$, 
\[
Z_t \equiv \EE\left[ \sup_{0<s<t} |E_s|^p \right]
\]
satisfies the inequality
\[
  Z_t \leq c_1\, h^{p/2} + c_2 \int_0^t Z_s \ \D s,
\]
from which it follows that $Z_t = O(h^{p/2})$ due to Gr\"onwall's inequality.
\end{proof}

\section{Extensions}

\subsection{Vector SDEs and vector parameters}

The analysis extends naturally to cases in which $S_t$ and
$\theta$ are both vectors. Thus, in the most general case
we are interested in computing matrices and tensors such as
\[
\frac{\partial (S_t)_i}{\partial \theta_j}, ~~~
\frac{\partial^2 (S_t)_i}{\partial \theta_j\partial \theta_k}
\]
where the subscripts $i, j, k$ refer to the components of $S_t$
and $\theta$.  The analysis does not change substantially, the
notation simply becomes much more cumbersome.

\subsection{Higher order sensitivities}

Higher order path sensitivities are of interest to the author in
connection with work extending the original MLMC research of
Heinrich on parametric integration \cite{heinrich01}. In addition,
second order sensitivities are potentially of interest in finance
applications when computing second order Greeks using a conditional
expectation technique for the final timestep to smooth the payoff
\cite{glasserman04}.

Differentiating the original scalar SDE a second time gives the
second order path sensitivity SDE
\[
  \D \ddS_t = (\dda_t + 2\da'_t\dS_t + a''_t (\dS_t)^2 + a'_t \ddS_t)\, \D t
            + (\ddb_t + 2\db'_t\dS_t + b''_t (\dS_t)^2 + b'_t \ddS_t)\, \D W_t.
\]
Continuing this, if $a(\theta,S)$ and $b(\theta,S)$ are both
$k$-times differentiable then by induction it can be proved that
the $k$-th order sensitivity equation has the form
\[
  \D S_t^{(k)} = ( a_t^{(k)} + a'_t\, S_t^{(k)})\, \D t
              + ( b_t^{(k)} + b'_t\, S_t^{(k)})\, \D W_t,
\]
where $S_t^{(k)}\equiv \partial^k S_t/\partial \theta^k$
and $a_t^{(k)}$ is a sum of terms of the form
\[
  \frac{\partial^{i+j} a}{\partial \theta^i \partial S^j} \ 
  \prod_{l=1}^{k-1} (S_t^{(l)})^{q_l}
\]
with positive integers $i, j, q_l$ satisfying $2\leq i{+}j\leq k$
and $\sum_{l=1}^{k-1}q_l = j$, and $b_t^{(k)}$ is a similar summation.

The Euler-Maruyama discretisation of this SDE is again equivalent
to the $k$-th order derivative of the Euler-Maruyama discretisation
of the original SDE.
The numerical analysis proceeds inductively, proving that if
all of the derivatives of $a$ and $b$ up to the $k$-th order
are uniformly bounded, and there are constants
$c_p^{(1,j)}, c_p^{(2,j)}, c_p^{(3,j)}$ for all $j<k$ such that
\begin{eqnarray*}
  \EE\left[\sup_{0<t<T} |S_t^{(j)}|^p \right] &\leq & c_p^{(1,j)},   \\[0.1in]
  \EE\left[\sup_{t_0<s<t} |S_s^{(j)} {-} S_{t_0}^{(j)}|^p \right] &\leq & c_p^{(2,j)}\, (t{-}t_0)^{p/2},  \\[0.1in]
  \EE\left[ \sup_{0<t<T} |\hS_t^{(j)} {-} S_t^{(j)}|^p \right] &\leq & c_p^{(3,j)}\, h^{p/2},
\end{eqnarray*}
then there are constants $c_p^{(1,k)}, c_p^{(2,k)}, c_p^{(3,k)}$ such that
similar bounds hold for $S_t^{(k)}$ and $\hS_t^{(k)}$.
The critical step in the analysis is the bounding of terms such as
$\EE[\, |a_t^{(k)} {-}  a_\ut^{(k)}|^p]$ and
$\EE[\, |a_\ut^{(k)} {-}  \ha_\ut^{(k)}|^p]$ which requires the
following simple lemma.

\begin{lemma}
  If $u_i, v_i$ $i=1, 2, \ldots k$ are scalar random variables,
  and for any $p\geq 2$ there are finite constants $C_p$, $D_p$
  such that
\[
\EE[\, |u_i|^p ] \leq C_p, ~~~
\EE[\, |v_i|^p ] \leq C_p, ~~~
\EE[\, |u_i{-}v_i|^p ] \leq D_p
\]
for all $i$, then
\[
  \EE\left[\ \left|\prod_{i=1}^k u_i - \prod_{i=1}^k v_i\right|^p\, \right] \leq
  k^p\, C_{pk}^{1-1/k}\, D_{pk}^{1/k}
\]
\end{lemma}
\begin{proof}
When $k=2$, $u_1 u_2 - v_1 v_2 = (u_1{-}v_1) u_2 + v_1 (u_2{-}v_2)$.
This generalises to 
\[
  \prod_{i=1}^k u_i - \prod_{i=1}^k v_i
  = \sum_{j=1}^k \left\{ \left(\prod_{i=1}^{j-1}
      v_i\right) (u_j-v_j) \left(\prod_{i=j+1}^{k}  u_i\right) \right\}
\]
By Jensen's inequality we have
\[
  \left|\prod_{i=1}^k u_i - \prod_{i=1}^k v_i\right|^p
  \leq k^{p-1} \sum_{j=1}^k  \left\{ \left( \prod_{i=1}^{j-1}  |v_i|^p \right)
  |u_j-v_j|^p  \left( \prod_{i=j+1}^{k} |u_i|^p \right) \right\}.
\]
For each $j$, H\"older's inequality gives
\begin{eqnarray*}
\lefteqn{ \EE\left[ \left( \prod_{i=1}^{j-1}  |v_i|^p \right)
  |u_j-v_j|^p  \left( \prod_{i=j+1}^{k} |u_i|^p \right) \right]}
  \\ &\leq&
\left( \prod_{i=1}^{j-1} \EE\left [ |v_i|^{pk}\right] \right)^{1/k}
\EE\left [|u_j-v_j|^{pk}\right]^{1/k}
\left( \prod_{i=j+1}^{k} \EE\left [|u_i|^{pk}\right] \right)^{1/k},
\end{eqnarray*}
and hence we obtain the desired result.
\end{proof}

%

\section{Conclusions}

This note has filled a gap in the stochastic numerical analysis
literature by proving the strong convergence of path sensitivity
approximations which do not satisfy the usual conditions assumed
for the analysis of Euler-Maruyama approximations.
The same order of strong convergence applies for higher order
sensitivities, provided the required drift and diffusion derivatives
exist and are bounded.

It is conjectured that the analysis in this note can be extended
to other discretisations such as the first order Milstein scheme,
but this is a topic for future analysis.

\bibliographystyle{plain}
\bibliography{../../bib/mc,../../bib/mlmc}

\end{document}